\documentclass[a4paper,oneside,12pt]{amsart}
\usepackage{amssymb}
\usepackage{mathrsfs}
\usepackage{amssymb,latexsym}
\usepackage[dvips]{graphicx}
\usepackage{psfrag}
\usepackage{hyperref}

\theoremstyle{plain}
\newtheorem{theorem}{Theorem}[section]

\newtheorem{lemma}[theorem]{Lemma}

\newtheorem{corollary}[theorem]{Corollary}

\numberwithin{equation}{section}

\begin{document}

\title[A Theorem of Stafford]
      {On a Theorem of Stafford}

\author{Napoleon Caro${}^{(1)}$}\thanks{${}^{(1)}$ Departamento de Matem{\'a}tica,
Instituto de Ci{\^e}ncias Matem{\'a}ticas e de Computa\c{c}{\~a}o, Universidade de S{\~a}o Paulo, Caixa Postal 668, CEP
13560-970, S{\~a}o Carlos, SP, Brazil.  \\ \textit{E-mail address}: napoct@icmc.usp.br}

\author{Daniel Levcovitz${}^{(2)}$}\thanks{${}^{(2)}$ Departamento de Matem{\'a}tica,
Instituto de Ci{\^e}ncias Matem{\'a}ticas e de Computa\c{c}{\~a}o, Universidade de S{\~a}o Paulo, Caixa Postal 668, CEP
13560-970, S{\~a}o Carlos, SP, Brazil.  \\ \textit{E-mail address}: lev@icmc.usp.br}


\maketitle

 \vspace{24pt}

 \begin{abstract}
 In [6] Stafford proved that every left or right ideal of the Weyl algebra $A_n(K)=K[x_1,... , x_n]\langle\partial_1,... ,\partial_n \rangle$ ($K$ a field of characteristic zero) is  generated by two  elements. Consider the ring $D_n :=K[[x_1,..., x_n]]\langle\partial_1,...\partial_n \rangle$  of differential operators over the ring of formal power series $K[[x_1,...,x_n]]$. In this paper we prove that every left or right ideal of the ring \linebreak
 $E_n :=K((x_1,...,x_n))\langle\partial_1,...,\partial_n \rangle$ of differential operators over the field of formal  Laurent series $K((x_1,...,x_n))$ is generated by two elements. The same is true for the ring of differential operators over the convergent \\ Laurent series $\mathbb{C}\{\{x_1,...,x_n\}\}.$ This is in accordance  with the conjecture that says that in a (noncommutative) noetherian simple  ring, every left or right  ideal is generated by two elements.
   \end{abstract}

 \medskip
 \medskip

 \thispagestyle{empty}
 \section{Introduction}

 In [6] Stafford proved that every left or right ideal of the Weyl algebra $A_n(K)=K[x_1,... , x_n]\langle\partial_1,... ,\partial_n \rangle$ ($K$ a field of characteristic zero) is  generated by two  elements. It would be interesting to have a similar  result for the ring $D_n :=K[[x_1,..., x_n]]\langle\partial_1,...\partial_n \rangle$  of differential operators over the ring of formal power series $K[[x_1,...,x_n]]$.   In this paper we prove that every left or right ideal of the ring $E_n :=K((x_1,...,x_n))\langle\partial_1,...,\partial_n \rangle$ of differential operators over the field of formal Laurent series $K((x_1,...,x_n))$ is generated by two elements. The same is true for the ring of differential operators over the convergent Laurent series $\mathbb{C}\{\{x_1,...,x_n\}
 \} $ (this is the field of fractions of the domain $\mathcal{O}_n$ of germs of convergent complex power series around the origin of $\mathbb{C}^n$). Note  that this is in accordance  with the conjecture that says that in a (noncommutative) noetherian simple  ring , every left or right  ideal is generated by two elements.

 The main difficulty here is that the symmetry between the $x$'s and the $\partial$'s, present in the Weyl algebra and which is fundamental in the prove of Stafford's theorem, is broken in the ring $D_n$. Since we are working over the base ring of power series, infinite sums in the $x$'s are permitted while only finite sums in the $\partial$'s occur. To deal with this problem we appeal to Weierstrass' preparation theorem to put an element $v \in D_n$ in a appropriate form (see lemma 4.1). In general, however, we follow the proof of Stafford's theorem in Bjork's book ([1]) and we do the necessary  modifications to deal with the power series case.

 \section{Basic Notations}

Let $K$ be a field of characteristic zero and $T$ be a skew field of characteristic zero.
By $T[[x]]$ we denote the ring of power series in one indeterminate with coefficients in $T$ and by $T((x))$ its quotient skew field.
In paragraph 3 we will be interested in the ring $S=T((x))\langle\partial_x \rangle$, the ring of differential operators with coefficients in $T((x))$. This paragraph is mostly concerned with free $S$-modules of finite rank over $S$.

In paragraph 4 the following notation will be used . For each $0\leq r \leq n$  let  $D_r=K[[x_1,...,x_r]]\langle\partial_1,...,\partial_r \rangle$ be the ring of differential operators over the ring of formal power series $K[[x_1,...,x_r]]$ and let  $F_r$ be its quotient ring. $F_r$ exists because $D_r$ is an Ore domain. Since $x_{r+1},...,x_n$ commute with the elements in $F_r$ we also get the division ring $F_r((x_{r+1},...,x_n))$  which by definition is the quotient ring of $D_r[[x_{r+1},...,x_n]]$.
Also, for each $0\leq r \leq n$,  we put $R_r=F_r((x_{r+1},...,x_n))\langle\partial_{r+1},...,\partial_n \rangle$. Of course if $r=n$, then $R_n=F_n$.

\section{The ring $S=T((x))\langle\partial_x \rangle$}
Let us explore some properties of $S=T((x))\langle\partial_x \rangle$, the ring of linear differential operators with coefficients in rational expressions in $x$ over the skew field $T$.
Recall that $S$ is a noncommutative noetherian simple ring.

 Suppose that $F$ is a field contained in $T$. By $F[[x]]\langle\partial_x \rangle$ we denote the ring of differential operators with coefficients in the ring of power series $F[[x]]$.

\begin{lemma}
Let $0\neq \alpha\in S$.Then the $S$-module $S/S\alpha $ has  finite length.
\end{lemma}
\begin{proof} It follows immediately from the division algorithm on $S$. See also [1] lemma 8.8, page 27.
\end{proof}

\begin{lemma}

Let $ \delta_1,...,\delta_m $ be a set of  non zero elements in  $F[[x]]\langle\partial_x \rangle \subset S$. Let $0\neq \alpha\in S$ and  $S^{(m)}=S\varepsilon_1 + \cdots +S\varepsilon_m$ be a free $S-module$ of rank $m$ with basis $\varepsilon_1,..., \varepsilon_m$. Let $M$ be the  $S-submodule $ of $S^{(m)}$  generated by the set $\{\alpha\delta_1f\varepsilon_1 + \dots+\alpha\delta_mf\varepsilon_m / f\in \mathbb{Z}[[x]]\langle\partial_x \rangle\}  $.Then $M=S^{(m)}$.
\end{lemma}

\begin{proof}
To simplify the notation we put $\partial:= \partial_x$

Let us first observe that both the assumption and the conclusion are \linebreak  unchanged if the $m-$tuple $ \delta_1,...,\delta_m $ is replaced by an $m-$tuple $ \beta_1,...,\beta_m,
 $ where $\beta_i=\Sigma a_{ij}\delta_j$ and $(a_{ij})$ is an $m\times m$ invertible matrix with $a_{ij}\in F$. Of course, while we replace $ \delta_1,...,\delta_m $ by $\beta_1,...,\beta_m $ under a $F$-linear transformation $(a_{ij})$ we also replace the free generators $\varepsilon_1,...,\varepsilon_m$ of $S^{(m)}$ by $\zeta_1,...,\zeta_m$ where $\zeta_i=\Sigma b_{ij}\varepsilon_j ,$  $(b_{ij})=(a_{ij})^{-1}$.

Let $ord(\delta_i)$ be the $\partial-$order of $\delta_i$. We can assume that these $\partial$-orders decrease, ie, $ord(\delta_1)\geqq ord(\delta_2)\geqq \cdots \geqq ord(\delta_m)$. Hence there exists an integer $\omega$ and some $1\leq l\leq m$ such that $\omega=ord(\delta_1)= \cdots =ord(\delta_l)$, while \linebreak  $ord(\delta_i)< \omega$ if $i>l$.

If $1\leq i\leq l$ we can write $\delta_i=r_i + p_i(x)\partial ^\omega$, where $ord(r_i)<\omega$ and $p_i(x)\in F[[x]]$.
We can assume that $val(p_1)\leq val(p_2)\leq ...\leq val(p_l)$, where $val(p_i)$ is the usual valuation of the power series $p_i$. If $val(p_1)=val(p_2)=\mu$, then there exists some $t\in F$ such that $val(p_2-tp_1)> \mu. $ Replace $\delta_2$ by $\delta_2-t\delta_1 $ while $\delta_1,\delta_3,...,\delta_m$ are unchanged.
Then, after some $F-$linear transformations  we can assume that  $val(p_1)< val(p_2) < \cdots < val(p_l)$. With these normalizations in hand we begin to prove that
$\varepsilon_1 \in M$.

Let $k=ord(\alpha)$. Since $S=T((x))\langle\partial \rangle$ we can assume that $\alpha=\alpha_0 + \partial^k$ where $ord(\alpha_0)< k$.
If $1\leq i\leq l$ we have $\alpha\delta_i=p_i(x)\partial^{k+\omega} + \psi_i$ where $ord(\psi_i)<k+ \omega$ and if $l< i\leq m$,  $ord(\alpha\delta_i)< k+\omega $.

Now if $g\in S$ we put $g_1=[g,x]=gx-xg$ the commutator of $g$ and $x.$ Inductively, let  $g_{\nu+1}=[g_\nu, x]$. The element  $g_{\nu}$ is called the \textit{$\nu-$fold commutator of $g$ and $x$.} The $\nu-$ fold commutator of $\partial^\nu$ and $x$ is $\nu!$ for all positive integers $\nu$, while the $\nu-$ fold commutator of $\partial^s$ and $x$ is zero if $s< \nu$.

If we apply this to the elements $\alpha\delta_1,...,\alpha\delta_m$ we see that the $(k+\omega)-$fold commutator of $\alpha\delta_i$ and $x$ is $p_i(x)(k+\omega)!$ for all $1\leq i\leq l$ ,while they are zero if $l< i\leq m$.

The definition of $M$ implies that $M$ is stable under the $\nu-$ fold commutator with $x$, i.e., if $m\in M$ then the $\nu$-fold commutator of $m$ and $x$ is in $M$ for all positive integers $\nu$. In fact, note that if $f \in \mathbb{Z}[[x]]\langle \partial \rangle,$ then $[ \alpha \delta f, x]= \alpha \delta (fx)- x(\alpha \delta f).$

For $f=1\in \mathbb{Z}[[x]]\langle\partial \rangle\ $ we have that $a=\alpha\delta_1\varepsilon_1 + \cdots + \alpha\delta_m\varepsilon_m\in M$. If $v_1$ is the $(k+\omega)-$fold commutator of $a$ and $x$ divided by $(k+ \omega)!,$ then $v_1\in M$ and  $$v_1=p_1(x)\varepsilon_1+ \cdots+p_l(x)\varepsilon_l.$$

Now, if $h \in S,$ let $h_1=[h, \partial]= h\partial-\partial h$ be the commutator of $h$ with $\partial$. Inductively,  we define $h_{\nu +1}$ by $h_{\nu +1}=[h_\nu,\partial]$. The element  $h_\nu$ is called the \textit{$\nu -$fold commutator of $h$ and $\partial$}.
We observe that $M$ is stable under the \linebreak $\nu -$ fold commutator with $\partial$ for all positive integers $\nu$.  In fact, note that if $f \in \mathbb{Z}[[x]]\langle \partial \rangle,$ then $[\alpha \delta f, \partial ]=\alpha \delta (f \partial) - \partial(\alpha \delta f) .$

Since $v_1\in M$ we have that if $v_1^{(\mu)}$ is the $\mu -$ fold commutator of $v_1$ and $\partial$, where $\mu=val(p_1(x)) < \cdots < val(p_l(x))$, then $v_1^{(\mu)}$ is in $M$ and  $$v_1^{(\mu)}=p_1^{(\mu)}(x)\varepsilon_1 + p_2^{(\mu)} (x)\varepsilon_2 + \cdots +p_l^{(\mu)}(x)\varepsilon_l\in M,$$ where $p^{(\mu)}(x)$ denotes the usual $\mu$-derivative of a power series $p(x).$ Note that $u(x):=p_1^{(\mu)}(x)$ is a unit in $F[[x]]$ and $val(p_j^{(\mu)}(x))= val(p_j(x)) - \mu > 0,$ \linebreak $ j=2, \cdots, l.$

Define $v_2:= v_1 - p_1(x)(u(x))^{-1}v_1^{(\mu)}.$ Then $v_2 \in M $ and $$v_2= q_2(x)\varepsilon_2 + \cdots +q_l(x)\varepsilon_l,$$ where
$q_j(x)=p_j(x) - p_1(x)(u(x))^{-1}p_j^{(\mu)}(x), j=2, ..., l.$ A simple  calculation shows that $q_j(x)$ is non-zero with  $val(q_j(x))= val(p_j(x)),$ for all $j=2, ...,l$. Therefore  $val(q_2)< val(q_3)< \cdots < val(q_l).$

We now, repeat the previous argument using the commutator of $v_2$ and $\partial$ until we get $v_3= r_3(x)\varepsilon_3 + \cdots + r_l(x)\varepsilon_l \in M.$ Proceeding in this way we finally get $v_l(x)=\widetilde{u}(x)\varepsilon_l \in M,$ where $\widetilde{u}(x)$ is a unit in $F[[x]].$ Therefore $\varepsilon_l \in M.$

If $l=1$, we are done. If $l> 1$, since $v_{l-1} \in M$, we have that $\varepsilon_{l-1} \in M.$ Going backwards we get $\varepsilon_1 \in M$.

Restricting the attention to the $(m-1)$-tuple $\delta_2,...,\delta_m$ and the $S-module$ $S^{(m-1)}=S\varepsilon_2 + \cdots +S\varepsilon_m,$ the lemma follows by induction over $m$.
\end{proof}

\begin{lemma}
Let $ \delta_1,...,\delta_m $ be a set of non zero elements  in  $ F[[x]]\langle\partial_x \rangle \subset S$ and let $M$ be a $S-submodule $ of $S^{(m)}$ such that the $S-module$ $S^{(m)}/M $ has finite length. If $0\neq \alpha\in S$,  there exists some  $f\in \mathbb{Z}[[x]]\langle\partial_x \rangle$ such that $$S^{(m)}=M + S(\alpha\delta_1f\varepsilon_1 + \cdots +\alpha\delta_mf\varepsilon_m ).$$
\end{lemma}
\begin{proof} See [1], Lemma 8.10, page 28.
\end{proof}

\begin{corollary}
Let $0\neq \rho \in S$ and let $ \delta_1,..., \delta_m $ be a set of non zero elements  in  $F[[x]]\langle\partial_x \rangle \subset S$. Then there exists some $f\in \mathbb{Z}[[x]]\langle\partial_x \rangle$ such that $$ S^{(m)}=S^{(m)}\rho + S(\rho \delta_1 f \varepsilon_1 + \cdots + \rho \delta_m f  \varepsilon_m ) $$
\end{corollary}
\begin{proof} See [1], Corollary 8.11, page 29.
\end{proof}

\begin{lemma}
Let $ \delta_1,...,\delta_m $ be a set of non zero elements  in  $F[[x]]\langle\partial_x \rangle \subset S$ and let $0\neq \rho\in S$. Consider the free
$S-module$  $S^{m+1}=S\varepsilon_0 +S\varepsilon_1 + \cdots + S\varepsilon_m $ with basis $\varepsilon_0, \cdots, \varepsilon_m$. Then there exists some $f\in \mathbb{Z}[[x]]\langle\partial_x \rangle$ such that $$S^{(m+1)}=S^{(m+1)}\rho + S(\varepsilon_0 + \delta_1f\varepsilon_1 + \cdots +\delta_1f\varepsilon_m ) $$ .
\end{lemma}
\begin{proof} See [1], Lemma 8.13, page 29.
\end{proof}

\section{Lemmas for $D_r$ and $R_r$}

In lemma 4.4 we will need to apply Weierstrass' preparation theorem to an non-zero element $v \in D_n =K[[x_1,..., x_n]]\langle\partial_1,...\partial_n \rangle.$  We will then state a separate lemma to prepare this element.

\begin{lemma} Let $v \in D_n$ be a non zero element. For any $r,$ $0 \leq r \leq n-1$, $v$ can be written in the following form:

$$v=\omega_1\beta_1 G_1 + \cdots + \omega_m\beta_m G_m , $$ where $\omega_1,...,\omega_m\in K[[x_1,...,x_n]] $ are units  , $\beta_1,...,\beta_m \in K[[x_{r+1}]]\langle \partial_{r+1}\rangle$ and $G_1,...,G_m \in D(r+1):=K[[x_1,...,\hat{x_{r+1}},...,x_n]]\langle \partial_1,...,\hat{\partial_{r+1}},...,\partial_n\rangle$.
\end{lemma}
\begin{proof}

Since $v \in D_n$ is a non zero element, we can write $v$ as a finite sum $\Sigma_{\alpha} p_{\alpha}(x_1,...,x_n)\partial^{\alpha}$, where each $p_{\alpha}(x_1,...,x_n)\in K[[x_1,...,x_n]] $ and  $\partial^{\alpha}=\partial_1^{\alpha_1}\cdots\partial_n^{\alpha_n}$. Let $k_\alpha$ be the order of the series $p_\alpha$. After a suitable linear change of variables we can assume that $p_\alpha(0,...0,x_{r+1},0,...,0) \in K[[x_{r+1}]]$ is non zero for all $\alpha$ and has valuation $k_\alpha$ as a series in the variable $x_{r+1}$. In fact, since $K$  is an infinite field and the $p_{\alpha}$ are in a finite number there is a single
linear change of variables that works for all  $p_{\alpha}$.
( See [2], lemma 2, page 17 and the remark that follows this lemma).

Now let us fixed $\alpha$. Then the series $p_{\alpha}(x_1,...,x_n)$ can be written as \linebreak  $\Sigma_{i=0}^\infty q_i(x_1,...,\hat{x_{r+1}} ,...,x_n)x_{r+1}^i$, where $q_i(x_1,...,\hat{x_{r+1}}, ...,x_n)\in K[[x_1,...,\hat{x_{r+1}},...,x_n]] .$ By our assumptions,  $q_{k_{\alpha}}(x_1,...,\hat{x_{r+1}} ,...,x_n)$ is  a unit in $K[[x_1,...,\hat{x_{r+1}},...,x_n]]$.By  Weierstrass' preparation theorem (see [4], page 208)  we can write \linebreak
$$p_{\alpha}=u(b_0+b_1 x_{r+1} + \cdots + x_{r+1}^{k_\alpha}),$$
where $u$ is a unit in $K[[x_1,...,x_n]]$ and $b_j\in K[[x_1,...,\hat{x_{r+1}},...,x_n]]$. Therefore we have \linebreak
$p_{\alpha}\partial^{\alpha}=u(b_0+b_1 x_{r+1} + \cdots + x_{r+1}^{k_\alpha})\partial_1^{\alpha_1}\cdots \partial_{r+1}^{\alpha_{r+1}}\cdots\partial_n^{\alpha_n}=$
$u(\partial_{r+1}^{\alpha_{r+1}})(b_0\partial_1^{\alpha_1}...\hat{\partial_{r+1}^{\alpha_{r+1}}}
...\partial_n^{\alpha_n})+u(x_{r+1}\partial_{r+1}^{\alpha_{r+1}})(b_1\partial_1^{\alpha_1}...\hat{\partial_{r+1}^{\alpha_{r+1}}}
...\partial_n^{\alpha_n})+\cdots+ u(x_{r+1}^{k_{\alpha}}\partial_{r+1}^{\alpha_{r+1}})(\partial_1^{\alpha_1}...\hat{\partial_{r+1}^{\alpha_{r+1}}}
...\partial_n^{\alpha_n})$,\linebreak
which has the desired form. Since $v=\sum_{\alpha}p_{\alpha}\partial^{\alpha}$ we are done.

\end{proof}

The following lemmas 4.3 and 4.4 are equivalent. We will then prove just the second. To prove it we will need the following result.
\begin{lemma} Let $0\leq r\leq n-1$ and let $q$ be a non zero element of $D_r[[x_{r+1}, \cdots, x_n]]$ and let $a_1, ..., a_t$  be a finite set in $D_n$.
Then there exists some $\rho \in D_r[[x_{r+1}, \cdots, x_n]]$, $\rho \neq 0$, such that $\rho a_j \in D_nq,$ for each $j=1,...,t$.
\end{lemma}
\begin{proof} See [1], lemma 8.5, page 26.
\end{proof}

\begin{lemma}
Let $0\leq r\leq n-1$ and let $0\neq q\in D_{r+1}[[x_{r+2},...,x_n]].$ If $u$ and $v$ are two elements in $D_n$ with $v\neq 0$, then there exists $f\in D_n $ and $Q_r\in D_r[[x_{r+1},...,x_n]]$ such that $$Q_r\in D_nq + D_n(u+ vf).$$
\end{lemma}

Recall that $R_r=F_r((x_{r+1},...,x_n))\langle\partial_{r+1},...,\partial_n \rangle,$ where $F_r$ is the quotient ring of $D_r=K[[x_1,...,x_r]]\langle\partial_1,...,\partial_r \rangle$.

\begin{lemma}
Let $0\leq r\leq n-1$ and let $0\neq q\in D_{r+1}[[x_{r+2},...,x_n]]$ and let $u$ and $v\in D_n$ with $v\neq 0$. Then there exist some $f\in D_n$ such that $$ R_r=R_rq+R_r(u+vf).$$
\end{lemma}

\begin{proof}

Let us fix $r$ such that $0 \leq r \leq n-1$. Since $v\neq 0$, by lemma 4.1  we can write $$v=\omega_1\beta_1 G_1 + \cdots+ \omega_m\beta_m G_m ,$$
 where  $\omega_1,...,\omega_m\in K[[x_1,...,x_n]] $ are units,  $\beta_1,...,\beta_m \in K[[x_{r+1}]]\langle \partial_{r+1}\rangle$ and $G_1,...,G_m \in D(r+1)=K[[x_1,...\hat{x_{r+1}} ...,x_n]]\langle \partial_1,...,\hat{\partial_{r+1}},...,\partial_n\rangle$.

  Let $\delta_i=\omega_i\beta_i$, for each $1\leq i\leq m$. Then $\delta_i\in K[[x_1,...,x_n]]\langle\partial_{r+1}\rangle$. To apply lemma 3.5, we observe that $K[[x_1,...,x_n]]= (K[[x_1,...\hat{x_{r+1}},...,x_n]])[[x_{r+1}]]  \subset F[[x_{r+1}]]$, where $F$ is the quotient field of $K[[x_1,...,\hat{x_{r+1}},...,x_n]]$. Therefore $\delta_i\in F[[x_{r+1}]]\langle\partial_{r+1}\rangle$, $F$ a field of characteristic zero. If $T=F_r((x_{r+1},...,x_n))$, then $T$ is a skew field and $F\subset T$.

 Since $v\neq 0$, we have that some $G_i\neq 0$. The ring $D(r+1) $ is simple, which implies the 2-sided ideal generated by $G_1,...,G_m$ is the whole ring $D(r+1)$. This gives finite sets $a_1,...,a_l$ and $b_1,...,b_l$ in $D(r+1)$ such that

 $$1=\Sigma_{j=1}^{m} \Sigma_{\nu = 1}^{l} b_\nu G_j  a_\nu$$ and hence $D(r+1)=\Sigma\Sigma D(r+1) G_j a_\nu$. Identifying $D(r+1)$ with a subring of $R_r$ we conclude that $R_r=\Sigma\Sigma R_r G_j a_\nu$.

 At this stage we need the following

 \emph{Claim:} To each $m-$tuple $B_1,...,B_m$ in $D(r+1)$ there exists some \linebreak  $f\in Z[[x_{r+1}]]\langle\partial_{r+1}\rangle$ such that
 $$R_rq + R_r u + R_rB_1+\cdots +R_rB_m =R_rq +  R_r(u+ \delta_1 f B_1 + \cdots + \delta_m f B_m).$$
 In fact, since $0\neq q\in D_{r+1}[[x_{r+2},...,x_n]]$, it follows from lemma 4.2 that there exists some $0\neq\rho\in D_{r+1}[[x_{r+2},...,x_n]]$ such that $\rho B_j\in D_n q $ for all $j=1,...,m$ and also $\rho u\in D_n q $.

 Let $S=T((x_{r+1}))\langle\partial_{r+1}\rangle $, then $0\neq\rho\in S$. Using the lemma $3.5$ we get some $f\in \mathbb{Z}[[x_{r+1}]]\langle\partial_{r+1}\rangle$ such that $S^{(m+1)}=S^{(m+1)}\rho + S(\varepsilon_0 + \delta_1f\varepsilon_1 + \cdots +\delta_1f\varepsilon_m )$.
 Since $S$ is a subring of $R_r$, we have that $R_r^{(m+1)}=R_r^{(m+1)}\rho + R_r(\varepsilon_0 + \delta_1f\varepsilon_1 +\cdots +\delta_1f\varepsilon_m )$.

 Considerer the $R_r-$linear application $\pi :R_r^{(m+1)} \rightarrow R_r$ defined by $\pi(\varepsilon_0)=u$ and $\pi(\varepsilon_j)=B_j$ for each $1\leq j \leq m.$ Then the image of $\pi$ is $R_ru + R_r B_1 +...+R_r B_m \subseteq R_r $; but $\rho u,$ $\rho B_j \in D_nq.$ Then we have that $\pi(\rho \varepsilon_0)=\rho\pi(\varepsilon_0)=\rho u \in R_r q $ and $\pi(\rho \varepsilon_j)=\rho\pi(\varepsilon_j)=\rho B_j\in R_r q $ for each $1\leq j\leq m$. Therefore $\pi(R_r^{(m+1)}\rho) \subseteq R_r q$. Then
  $$R_rq + R_ru + R_rB_1 +\cdots +R_rB_m=R_rq+\pi(R_r^{(m+1)})$$
  $$\subseteq R_rq + \pi(R_r^{(m+1)}\rho + R_r(\varepsilon_0 +\delta_1f\varepsilon_1 +\cdots+\varepsilon_0 +\delta_mf\varepsilon_m))$$
  $$\subseteq R_rq + \pi(R_r^{(m+1)}\rho) + R_r(u +\delta_1 fB_1 +\cdots+\delta_m fB_m)$$
  $$\subseteq R_rq + R_r(u +\delta_1 fB_1 +\cdots+\delta_m fB_m)$$
and this proves the claim because the opposite inclusion is clear.

Now we apply the claim to $B_j=G_ja_1, j=1,...,m$. Then, there exists $f_1\in \mathbb{Z}[[x_{r+1}]]\langle\partial_{r+1}\rangle$ such that
$$R_rq + R_ru + R_rG_1a_1 +\cdots+R_rG_ma_1=R_rq + R_r(u +\delta_1 f_1 G_1 a_1 +\cdots+\delta_m f_1 G_m a_1).$$
Since $f_1\in \mathbb{Z}[[x_{r+1}]]\langle\partial_{r+1}\rangle $ and $G_j\in D(r+1)$ commute, we have that
 $$\delta_1 f_1 G_1 a_1 +\cdots+\delta_m f_1 G_m a_1=\delta_1  G_1f_1 a_1 +\cdots+\delta_m  G_mf_1 a_1=\omega_1\beta_1G_1f_1 a_1 + \cdots+ \omega_m\beta_m G_mf_1 a_1$$  $=vf_1a_1,$
since $v=\omega_1\beta_1 G_1 + \cdots+ \omega_m\beta_m G_m$.
 Then, $$R_rq + R_ru + R_rG_1a_1 +\cdots+R_rG_ma_1=R_rq + R_r(u + vf_1a_1).$$

Now we apply the claim again with $u$ replaced by $u +vf_1a_1$ and  $B_j=G_ja_2, j=1,...,m$. There exists $f_2\in \mathbb{Z}[[x_{r+1}]]\langle\partial_{r+1}\rangle$ such that $R_rq + R_r(u +vf_1a_1) + \sum R_r G_ja_2 =R_rq + R_r(u + vf_1a_1 + vf_2a_2)$. Using the previous equation we have $$R_rq + R_ru + \sum R_rG_ja_1 + \sum R_rG_ja_2= R_rq+R_r(u +vf_1a_1 + vf_2a_2).$$

In the next step we apply the claim with $B_j=G_ja_3, j=1,...,m$ and $u$ replaced by $u +vf_1a_1 +vf_2a_2$. After $l$ steps we have
$$R_rq + R_r(u + vf_1a_1 +\cdots+vf_la_l)=R_rq + R_ru + \sum\sum R_rG_ja_\nu =R_r.$$

Hence the lemma follows with $f=f_1a_1 +\cdots+f_la_l$.
\end{proof}

\section{The Principal Result}

\begin{lemma}
Let $a,b$ and $c$   non zero elements of $D_n$. For each $0\leq r\leq n$ there exist $q_r\in D_r[[x_{r+1},...,x_n]], q_r \neq 0$ and $d_r,e_r\in D_n$ such that $$q_rc\in D_n(a+d_rc) + D_n(b+e_rc).$$
\end{lemma}
\begin{proof} With $d_n=0$ and $e_n=0$ we see that the statement is true for $r=n$, since $D_n$ is a left Ore domain and $D_na \subset (D_na + D_nb).$

For $0\leq r\leq (n-1)$ the proof is by induction from $r+1$ to $r$. For this see the prove that Proposition 7.3(r+1) $\Rightarrow$ Proposition 7.3(r) in the book of Bj\"{o}rk [1, page 22], in which the lemma 7.5 should be replaced by our lemma 4.3.
\end{proof}

\begin{theorem}

Any left or right ideal in $K((x_1,...,x_n))\langle \partial_1,...,\partial_n\rangle$ can be ge-nerated by two elements. The same is true for the ring $\mathbb{C}\{\{x_1,...,x_n\}\}\langle \partial_1,...,\partial_n\rangle.$
\end{theorem}
\begin{proof}  The ring $E_n =K((x_1,...,x_n))\langle\partial_1,...,\partial_n \rangle$ is a noetherian ring. Therefore, it is enough to show that given
$a,b,c \in E_n$ there exists $d,e \in E_n$ such that $c \in E_n(a+dc) + E_n(b+ec)$ .

Take $n=0$ in the previous lemma. Then, there exists  $q_0 \in K[[x_1,..., x_n]] ,$ $ q_0 \neq 0$ and $d,e \in D_n$ such that $q_0c \in D_n(a+dc) + D_n(b+ec).$
Since \linebreak $D_n =K[[x_1,..., x_n]]\langle\partial_1,...,\partial_n \rangle$, then  $c \in E_n(a+dc) + E_n(b+ec)$.
\end{proof}



\begin{thebibliography}{}

\bibitem{bjork} Bj\"{o}rk, J.-E., {\it Rings of differential operators}, North-Holland Mathematical Library, \\ vol. 21,(1979).
North-Holland Publishing Co., Amsterdam, New York.
\bibitem{gunning} Gunning, R.-C., {\it Introduction to holomorphic functions of several variables, volume I}, Wadsworth and Brooks/Cole Mathematics Series, California (1990).

\bibitem{hillebrand}Hillebrand, A., Schmale, W., {\it Towards an effective version of a theorem of Stafford}. Effective
methods in rings of differential operators, J. Symbolic Comput. 32 ,(2001), 699–716.
\bibitem{lang} Lang.S., {\it Algebra}, Springer-Verlag, (2002) New York Inc.Revised Third Edition.
\bibitem{leykin} Leykin.A., {\it Algorithmic proofs of two theorems of Stafford}, J. Symbolic Comput. 38(2004), 1535–1550.
\bibitem{stafford} Stafford, J.T., {\it Module structure of Weyl algebras},
J. London Math. Soc. (2) 18 (1978), no. 3, 429--442.




\end{thebibliography}
\end{document}